\documentclass[11pt]{amsart}

\usepackage{amssymb}
\usepackage{graphicx}   
\usepackage{mathrsfs}
\usepackage{amsmath}
\usepackage{amsthm}
\usepackage{amsfonts}
\usepackage{latexsym}

\newtheorem{thm}{Theorem}[section]
\newtheorem{prop}[thm]{Proposition}
\newtheorem{lem}[thm]{Lemma}

\theoremstyle{definition}

\theoremstyle{claim}
\newtheorem{claim}[thm]{Claim}

\theoremstyle{remark}
\newtheorem{remark}[thm]{Remark}

\numberwithin{equation}{section}

\begin{document}

%%
%% The title of the paper goes here.  Edit to your title.
%%

\title{A Natural Min-Max Construction for Ginzburg-Landau Functionals}

\author{Daniel L. Stern}
\address{Department of Mathematics, Princeton University, 
Princeton, NJ 08544}
\email{dls6@math.princeton.edu}

%%
%% If there are three of more authors they are added in the obvious
%% way. 
%%

%%%
%%% The following is for the abstract.  The abstract is optional and
%%% if not used just delete, or comment out, the following.
%%%

\begin{abstract}
We use min-max techniques to produce a family of nontrivial solutions $u_{\epsilon}: M^n\to \mathbb{R}^2$ of the Ginzburg-Landau equation
$$\Delta u_{\epsilon}+\frac{1}{\epsilon^2}(1-|u_{\epsilon}|^2)u_{\epsilon}=0$$
on a given compact Riemannian manifold $M^n$, whose energy grows like $|\log\epsilon|$ as $\epsilon\to 0$. Building on the analysis of \cite{BBO}, we show that when the degree one cohomology $H^1_{dR}(M)=0$, the energy of these solutions concentrates on a nontrivial stationary, rectifiable $(n-2)$-varifold $V$.
\end{abstract}

\maketitle

%%
%% LaTeX can automatically make a table of contents.  This is done by
%% uncommenting the following:
%%

\tableofcontents

%%
%% A Theorem is stated by
%%

%\begin{thm} The square of any real number is non-negative.
%\end{thm}

% \begin{proof}
% Any real number $x$ satisfies $x>0$, $x=0$, or $x<0$.
% If $x=0$, then $x^2=0\ge 0$.  If $x>0$ then as a positive time a
% positive is positive we have $x^2=xx>0$.  If $x<0$ then $-x>0$ and so
% by what we have just done $x^2=(-x)^2>0$.  So in all cases $x^2\ge0$.
% \end{proof}

%%
%% A new section is started as follows:
%%

\section{Introduction}

\hspace{6mm} Let $(M,g)$ be a compact Riemannian manifold of dimension $n\geq 2$. Given a complex-valued map $u: M\to \mathbb{R}^2$, we define for $\epsilon>0$ the Ginzburg-Landau functionals
\begin{equation}\label{gledef}
E_{\epsilon}(u):=\int_Me_{\epsilon}(u)=\int_M\frac{1}{2}|du|^2+\frac{1}{\epsilon^2}W(u).
\end{equation}
Here, $W:\mathbb{R}^2\to \mathbb{R}$ is a smooth, bounded potential satisfying
\begin{equation}\label{w1}
W(z)=\frac{1}{4}(1-|z|^2)^2  \text{ for }|z|<2,
\end{equation}
\begin{equation}\label{w2}
W(z)\geq 2 \text{ for }|z|\geq 2, 
\end{equation}
and
\begin{equation}\label{w3}
\sup_{z\in \mathbb{R}^2}|DW(z)|<\infty.
\end{equation}
Critical points $u_{\epsilon}:M\to \mathbb{R}^2$ of the energy $E_{\epsilon}$ solve the Ginzburg-Landau equation
\begin{equation}\label{eleqn}
\Delta u_{\epsilon}=\frac{1}{\epsilon^2}DW(u_{\epsilon}).
\end{equation}

\hspace{6mm} Clearly, the global minimizers of $E_{\epsilon}$ are just the constant maps taking values in the unit circle. On a bounded domain $\Omega\subset\mathbb{R}^n$, we can find more interesting solutions of (\ref{eleqn}) by minimizing $E_{\epsilon}(u)$ among maps with fixed Dirichlet data
\begin{equation}\label{dirconst}
u|_{\partial\Omega}=h_{\epsilon}.
\end{equation}
When $\Omega\subset\mathbb{R}^2$ is a simply-connected planar domain, and $h_{\epsilon}$ is a fixed map
$$h:\partial\Omega\to S^1$$
of degree $d$, the asymptotic behavior of these minimizers $u_{\epsilon}$ as $\epsilon\to 0$ was characterized by the work of Bethuel-Brezis-H\'{e}lein \cite{BBH} and Struwe \cite{St}. Namely, they showed that (along some subsequence $\epsilon_j\to 0$), there exist $|d|$ points $a_1,\ldots,a_{|d|}\in \Omega$ such that
\begin{equation}
\lim_{\epsilon\to 0}\frac{e_{\epsilon}(u_{\epsilon})}{|\log\epsilon|}dx= \pi\cdot\Sigma_{j=1}^{|d|}\delta_{a_j}\text{ in }(C^0)^*,
\end{equation}
while
\begin{equation}
u_{\epsilon} \to u\text{ in }C^{\infty}_{loc}(\Omega\setminus \{a_1,\ldots,a_{|d|}\}),
\end{equation}
where $u:\Omega\to S^1$ is a weakly harmonic map with singularities at $\{a_1,\ldots,a_{|d|}\}$ (\cite{BBH},\cite{St}).  In particular, these results establish the variational theory of $E_{\epsilon}$ as a natural means for producing singular harmonic maps to $S^1$ in situations where finite-energy solutions aren't available.

\hspace{6mm} For solutions in higher dimensions, a still richer structure emerges, with connections to geometric measure theory. (We assume here some familiarity with the basic definitions and results of geometric measure theory, as found in \cite{F} and \cite{Si}; see especially \cite{Si} for the theory of varifolds.) For domains $\Omega\subset\mathbb{R}^n$, $n\geq 3$, Lin and Rivi\`{e}re studied minimizers of $E_{\epsilon}$ under boundary conditions $h_{\epsilon}:\partial\Omega\to D^2$ that approximate a map $\partial \Omega\to S^1$ with singularity along a fixed $n-3$-dimensional submanifold $S\subset\partial\Omega$ \cite{LR}. In a striking extension of the two-dimensional results, they showed that (along a subsequence) the measures
$$\mu_{\epsilon}:=\frac{e_{\epsilon}(u_{\epsilon})}{\pi |\log\epsilon|}dx$$
converge to the weight measure $\mu_T$ of an integral $(n-2)$-current $T\in \mathcal{I}_{n-2}(\Omega)$ solving the Plateau problem
\begin{equation}\label{plateau}
\partial T=S,\text{\hspace{3mm} }{\bf M}(T)\leq {\bf M}(T+\partial W)\text{ for all } W\in \mathcal{I}_{n-1}(\Omega),
\end{equation}
while, away from $spt(T)$, $u_{\epsilon}$ again converges to a harmonic map $u:\Omega\setminus spt(T)\to S^1$ \cite{LR}. The proof of this statement doesn't rely on the existence of a solution to (\ref{plateau}), so these results yield a new existence proof for the codimension-two Plateau problem via Ginzburg-Landau functionals \cite{LR}.

\hspace{6mm} In \cite{BBO}, Bethuel, Brezis, and Orlandi employed ideas from \cite{AS} and \cite{LR} to produce similar results for non-minimizing solutions $u_{\epsilon}$ of (\ref{eleqn}) with boundary data $h_{\epsilon}$ similar to that used in \cite{LR}. For such solutions, they showed that the normalized energy measures $\mu_{\epsilon}$ concentrate on a stationary, rectifiable varifold of codimension two, away from which the maps $u_{\epsilon}$ converge smoothly to a harmonic map to $S^1$. In particular, their results give us reason to hope that the variational theory of the Ginzburg-Landau functional could be used to produce nontrivial critical points of the $(n-2)$-area functional.

\hspace{6mm} In this paper, we introduce a natural min-max procedure for the Ginzburg-Landau energies to produce solutions on an arbitrary compact manifold whose energy concentration measures $\mu_{\epsilon}$ have mass bounded above and below:

\begin{thm}\label{mainthm} On any compact Riemannian manifold $(M^n,g)$, there exists a family of nontrivial solutions $u_{\epsilon}:M\to \mathbb{R}^2$ of the Ginzburg-Landau equations $(\ref{eleqn})$ satisfying energy bounds of the form
\begin{equation}\label{cbounds}
c|\log\epsilon|\leq E_{\epsilon}(u_{\epsilon})\leq C|\log\epsilon|
\end{equation}
for some positive constants $C=C(M),$ $c=c(M)$. 
\end{thm}

Moreover, when $M$ has vanishing degree one cohomology $H^1_{dR}(M)=0$, we show that the analysis of \cite{BBO} can be extended to arbitrary global solutions of (\ref{eleqn}) satisfying bounds of the form (\ref{cbounds}), to conclude that

\begin{thm}\label{thm2} If $H^1_{dR}(M)=0$, then $\exists$ a subsequence $\epsilon_j\to 0$ and a nontrivial stationary, rectifiable $(n-2)$-varifold $V$ on $M$ such that 
$$\mu_{\epsilon_j}:=\frac{e_{\epsilon_j}(u_{\epsilon_j})}{|\log\epsilon_j|}dv_g\to \|V\|.$$
\end{thm}

\begin{remark}\label{harmex}
When $H^1_{dR}(M)\neq 0$, it is no longer true that bounds of the form (\ref{cbounds}) yield compactness results for the solutions $u_{\epsilon}$, or $(n-2)$-rectifiability of the energy concentration measure. For instance, consider $M=S^1\times N$ endowed with the product metric, let $\epsilon_k=e^{-k^2}$, and let $u_{\epsilon_k}:M\to\mathbb{C}$ be given by $u_{\epsilon_k}(z,y)=(1-k^2\epsilon_k^2)^{1/2}z^k$. These $u_{\epsilon_k}$ then solve $(GL)_{\epsilon_k}$ with energy bounds of the form (\ref{cbounds}), but the energy concentration measures $\mu_{\epsilon_k}$ converge to a multiple of the volume measure $dv_g$ as $k\to \infty$. We nonetheless expect that the conclusion of Theorem \ref{thm2} will hold for our min-max solutions when $H^1_{dR}(M)\neq 0$, but such a result will necessarily rely in a nontrivial way on the min-max construction.
\end{remark}

\begin{remark}
As in \cite{BBO}, after establishing a positivity result for the $(n-2)$-density of the limiting measure, the concentration of energy on an $(n-2)$-rectifiable varifold in Theorem \ref{thm2} follows from Ambrosio and Soner's blow-up argument in \cite{AS}. In particular, it does not follow from our analysis that $V$ has integer density $\mathcal{H}^{n-2}$-a.e. For applications to geometric measure theory, it would be very interesting to extend the integrality results of \cite{LR} for minimizers to the min-max solutions constructed here.
\end{remark}

\hspace{6mm} These results are inspired in large part by Guaraco's min-max program for the elliptic Allen-Cahn equation--the scalar analog of (\ref{eleqn}) \cite{Gu}. Building on results of Hutchinson-Tonegawa \cite{HT} and Tonegawa-Wickramasekera \cite{TW}, it was shown in \cite{Gu} that real-valued solutions of (\ref{eleqn}) arising from a natural mountain-pass construction exhibit energy blow-up on a stationary, integral $(n-1)$-varifold, with singular set of Hausdorff dimension $\leq n-8$. In particular, the analysis in \cite{Gu} recovers the major results of the Almgren-Pitts min-max construction of minimal hypersurfaces \cite{A}, \cite{P}, while replacing a number of the original geometric measure theory arguments with (often simpler) pde methods. 

\hspace{6mm} The conclusions of \cite{Gu} are particularly intriguing in light of recent applications of the min-max theory of minimal hypersurfaces to some long-standing problems in geometry, such as Marques and Neves's resolution of the Willmore Conjecture \cite{MN2}, or their proof that manifolds of positive Ricci curvature contain infinitely many minimal hypersurfaces \cite{MN1}. As a natural regularization of the Almgren-Pitts theory, the Allen-Cahn min-max serves as a bridge between these kinds of results and questions in semilinear pde. In \cite{GG}, for example, Gaspar and Guaraco draw on this relationship by adapting the arguments in \cite{MN1} to the Allen-Cahn setting, obtaining a number of new results about the solution space of semilinear pdes of this type.

\hspace{6mm} Our results here suggest that the min-max theory of Ginzburg-Landau functionals may provide a similar regularization for the Almgren-Pitts min-max in codimension two. To this end, it would be desirable to extend our results by removing the cohomological constraint in Theorem \ref{thm2} and investigating the integrality of the energy-concentration varifold. Da Rong Cheng has informed us that he has independently obtained the result of Theorem \ref{thm2}.

\section*{Acknowledgements} I would like to thank my advisor Fernando Cod\'{a} Marques for his constant support and many helpful conversations. The author is partially supported by NSF grants DMS-1502424 and DMS-1509027.

%%
%% A subsection is started by
%%

%\subsection{Blah}

\section{The Min-Max Procedure}

\hspace{6mm} Our basic method for constructing critical points of $E_{\epsilon}$ is a natural extension to codimension two of Guaraco's mountain pass construction in \cite{Gu}. Namely, we employ a simple two-parameter min-max procedure (following the presentation in \cite{Gh}) to obtain nontrivial critical points of $E_{\epsilon}$ on $M$.

\hspace{6mm} Let $\varphi$ be a $C^1$ functional on a Banach space $X$, and suppose $X$ splits into a sum
$$X=Y\oplus Z,$$
where $\dim(Y)=k<\infty$.  
Denote by $B_Y$ the closed unit $k$-ball
$$B_Y:=\{u\in Y \mid \|u\|\leq 1\},$$
and let 
$$S_Y:=\{u\in Y\mid \|u\|=1\}$$ 
be its boundary $(k-1)$-sphere. Let $\Gamma$ be the collection of maps
\begin{equation}
\Gamma:=\{F\in C^0(B_Y,X) \mid F|_{S_Y}=Id|_{S_Y}\},
\end{equation}
and $c$ the associated min-max constant
\begin{equation}
c:=\inf_{F \in \Gamma}\max_{y\in B_Y}\varphi(F(y)).
\end{equation}

\hspace{6mm }For any family $F \in \Gamma$, given a projection $P_Y:X\to Y$, we can apply elementary degree theory to the map
$$P_Y\circ F: B_Y\to Y$$
to conclude that $P_Y\circ F$ must vanish somewhere, so that $F(y)\in Z$ for some $y \in B_Y$. If we have also an estimate of the form
\begin{equation}\label{zbarrier}
\inf \varphi(Z)>\sup \varphi(S_Y),
\end{equation}
it then follows from general versions of the min-max theorem (e.g., Theorem 3.2 in \cite{Gh}) that 

\begin{thm}\label{minmaxthm} For any sequence $F_j\in \Gamma$ such that
\begin{equation}\lim_{j\to \infty}\sup_{y\in B_Y}\varphi(F_j(y))=c,
\end{equation}
there exists a sequence $u_j\in X$ such that
\begin{equation}\label{ps1}
\lim_{j\to \infty}\varphi(u_j)=c,
\end{equation}
\begin{equation}\label{ps2}
\lim_{j\to\infty}\|d\varphi(u_j)\|=0,
\end{equation}
and
\begin{equation}\label{ps3}
\lim_{j\to\infty}dist(u_j,F_j(B_Y))=0.
\end{equation}
\end{thm}

\hspace{6mm} It's not difficult to see how the Ginzburg-Landau energy $E_{\epsilon}$ fits into this framework. By our assumptions on the structure of $W$, $E_{\epsilon}$ is a $C^1$ functional on the Sobolev space $H^1(M,\mathbb{R}^2)$, with derivative $E_{\epsilon}'$ given by
$$\langle E_{\epsilon}'(u), v\rangle=\int_M\langle du,dv\rangle+\epsilon^{-2}\langle DW(u),v\rangle.$$
If we consider the natural splitting
$$H^1(M,\mathbb{R}^2)=\mathbb{R}^2\oplus Z$$
of $H^1(M,\mathbb{R}^2)$ into the constant maps (identified with $\mathbb{R}^2$) and the orthogonal complement
$$Z:=\{u \in H^1(M,\mathbb{R}^2)\mid \int_Mu=0\in \mathbb{R}^2\},$$
then we note that the unit circle $S^1\subset\mathbb{R}^2$ in the $\mathbb{R}^2$ factor is precisely the subset of $H^1(M,\mathbb{R}^2)$ on which $E_{\epsilon}$ vanishes. Thus, to apply Theorem \ref{minmaxthm} to obtain a nice min-max sequence for $E_{\epsilon}$, it is enough to establish an estimate of the form (\ref{zbarrier}): namely, we need to show that
\begin{equation}\label{zbarrier2}
\inf_{u\in Z}E_{\epsilon}(u)>0.
\end{equation}

\hspace{6mm} Such an estimate is easy to obtain: The Poincar\'{e} inequality furnishes us with a constant $\lambda_1(M)>0$ such that
$$\int_M|du|^2\geq \lambda_1(M)\int_M|u|^2\text{ for all }u\in Z;$$
hence, for any $u \in Z$, we find that
\begin{eqnarray*}
E_{\epsilon}(u)&=&\int_M\frac{|du|^2}{2}+\frac{W(u)}{\epsilon^2}\\
&\geq &\int_M \frac{\lambda_1}{2}|u|^2+\frac{W(u)}{\epsilon^2}\\
&\geq &\int_{\{|u|\geq 1/2\}}\frac{\lambda_1}{2}|u|^2+\int_{\{|u|<1/2\}}\frac{W(u)}{\epsilon^2}\\
&\geq &\min \{\frac{\lambda_1(M)}{8},\frac{W(1/2)}{\epsilon^2}\}\cdot \frac{1}{2}vol(M)>0.
\end{eqnarray*}

\hspace{6mm} Thus, (\ref{zbarrier2}) holds, and we are indeed in a position to apply the min-max theorem \ref{minmaxthm}. That is, letting $D\subset\mathbb{R}^2$ denote the closed unit disk, and setting
\begin{equation}\label{gammadef}
\Gamma(M):=\{F\in C^0(D, H^1(M,\mathbb{R}^2))\mid F(y)\equiv y\text{ for }y\in S^1\},
\end{equation}
and
\begin{equation}\label{cdef}
c_{\epsilon}(M):=\inf_{F\in \Gamma(M)}\max_{y\in D}E_{\epsilon}(F(y)),
\end{equation}
we can extract from any minimizing sequence of families
\begin{equation}\label{minfam}
F_j\in \Gamma(M),\text{\hspace{3mm}}\lim_{j\to\infty}\max_{y\in D}E_{\epsilon}(F_j(y))=c_{\epsilon}
\end{equation}
a min-max sequence $u_j$ satisfying (\ref{ps1})-(\ref{ps3}). 

\hspace{6mm} Given any family $F \in \Gamma(M)$, we can apply the nearest-point retraction $\Phi: \mathbb{R}^2\to D$ to obtain a new family $\tilde{F}:=\Phi\circ F\in \Gamma$; it is clear that $Lip(\Phi)=1$ and $W\circ \Phi \leq W$, and therefore
$$E_{\epsilon}(\tilde{F}(y))\leq E_{\epsilon}(F(y))\text{ for each }y\in D.$$
In particular, starting from any minimizing sequence of families $F_j$ as in (\ref{minfam}), we can apply $\Phi$ to obtain a new minimizing sequence $\tilde{F}_j$ satisfying
\begin{equation}\label{modbound}
\|\tilde{F}_j(y)\|_{\infty}\leq 1.
\end{equation}
If $u_j\in H^1(M,\mathbb{R}^2)$ is a min-max sequence satisfying (\ref{ps1})-(\ref{ps3}) with respect to $\tilde{F}_j$, then the bound (\ref{modbound}), together with (\ref{ps3}), implies that $u_j$ is bounded in $L^2$, and since 
$$\lim_{j\to\infty}\frac{1}{2}\int_M|du_j|^2\leq \lim_j E_{\epsilon}(u_j)=c_{\epsilon},$$
it follows that $u_j$ is bounded in the full $H^1$ norm. 

\hspace{6mm} It is a simple and well known fact (see, e.g., \cite{Gu}, \cite{JSt}) that functionals of Ginzburg-Landau type satisfy the Palais-Smale condition along bounded sequences: that is, if 
$$\sup_j\|u_j\|_{H^1}<\infty\text{ and }\lim_{j\to\infty}\|E_{\epsilon}'(u_j)\|=0,$$
then $u_j$ contains a strongly convergent subsequence $u_j\to u$, whose limit necessarily satisfies
$$E_{\epsilon}'(u)=0\text{ and }E_{\epsilon}(u)=\lim_{j\to\infty}E_{\epsilon}(u_j).$$
Applying this fact to the min-max sequence of the previous paragraph, we obtain our basic existence result:

\begin{prop}\label{existence} For any $\epsilon>0$, there exists a critical point $u_{\epsilon}\in H^1(M,\mathbb{R}^2)$ of $E_{\epsilon}$ such that
\begin{equation}\label{minmaxenergy}
E_{\epsilon}(u_{\epsilon})=c_{\epsilon}(M)>0,
\end{equation}
and
\begin{equation}\label{modbound2}
\|u_{\epsilon}\|_{\infty}\leq 1.
\end{equation}
\end{prop}
To prove Theorem \ref{mainthm}, it remains to establish the energy estimates
$$0<\liminf_{\epsilon\to 0}\frac{c_{\epsilon}(M)}{|\log\epsilon|}\leq \limsup_{\epsilon\to 0}\frac{c_{\epsilon}(M)}{|\log\epsilon|}<\infty.$$

\section{Lower Bounds on the Energies}

\hspace{6mm} Since we've shown that the min-max constants $c_{\epsilon}(M)$ are positive critical values of the energy $E_{\epsilon}$, one obvious way to obtain lower bounds for $c_{\epsilon}(M)$ is to find lower bounds for the energy of arbitrary nontrivial solutions of (\ref{eleqn}). Simple examples show, however, that such estimates will not in general yield lower bounds of the desired form. 

\hspace{6mm} Consider, for instance, $M=S^1\times N$ endowed with the product metric, and let $p:M\to S^1$ be the obvious projection. For each $\epsilon\in (0,1)$, it's easy to check that the maps
$$p_{\epsilon}=(1-\epsilon^2)^{\frac{1}{2}}\cdot p$$
satisfy (\ref{eleqn}), while their energies $E_{\epsilon}(p_{\epsilon})$ stay uniformly bounded as $\epsilon\to 0$. (As an aside, we note that the maps $p_{\epsilon}$ also satisfy $\int_Mp_{\epsilon}=0$, so, in contrast to the situation for the Allen-Cahn min-max \cite{Gu}, we can't hope to establish the desired energy blow-up by proving lower bounds for $E_{\epsilon}$ over maps of zero average.)

\hspace{6mm} The problem in the example above comes from the existence of a nontrivial harmonic map $M\to S^1$. Recall that (modulo rotation) smooth harmonic maps to $S^1$ are in one-to-one correspondence with harmonic one-forms representing integer cohomology classes in $H^1_{dR}(M)$. In particular, when $H^1_{dR}(M)=0$ there are no nontrivial harmonic maps $M\to S^1$, and in this case, we find the following:

\begin{lem}\label{toplem} If $H_{dR}^1(M)=0$, then for any family $u_{\epsilon}$ of nontrivial solutions of $(\ref{eleqn})$, we have the lower energy bound
\begin{equation}\label{goodbound}
\liminf_{\epsilon\to 0}\frac{E_{\epsilon}(u_{\epsilon})}{|\log\epsilon|}>0.
\end{equation}
\end{lem}

\begin{proof} To begin, we show that any nontrivial solution $u_{\epsilon}$ must vanish somewhere. To see this, suppose $u_{\epsilon}$ solves (\ref{eleqn}), and that
\begin{equation}\label{nonvanishing}
|u_{\epsilon}(x)|>0\text{ for all }x\in M.
\end{equation}
Let $ju_{\epsilon}$ denote the pull-back of the one-form $r^2d\theta\in \Omega^1(\mathbb{R}^2)$ by $u_{\epsilon}$--i.e.,
\begin{equation}\label{jdef}
ju_{\epsilon}:=u_{\epsilon}^1du_{\epsilon}^2-u_{\epsilon}^2du_{\epsilon}^1.
\end{equation}
Computing the divergence of $ju_{\epsilon}$ and applying (\ref{eleqn}), we arrive at
\begin{equation}\label{divj}
d^*ju_{\epsilon}=0,
\end{equation}
a fundamental fact for solutions of (\ref{eleqn}). That is, for any $\psi \in C^{\infty}(M)$, we have
\begin{equation}\label{intdivfree}
\int_M\langle ju_{\epsilon},d\psi\rangle=0.
\end{equation}

\hspace{6mm} In light of (\ref{nonvanishing}), consider the smooth map 
$$\phi:=\frac{u_{\epsilon}}{|u_{\epsilon}|}:M\to S^1,$$
and observe that the pullback $\phi^*d\theta$ is a closed one-form; hence, by our assumption on the cohomology of $M$, there exists some $\psi\in C^{\infty}(M)$ such that 
\begin{equation}\label{exact}
\phi^*d\theta=d\psi.
\end{equation}
On the other hand, we also note that
$$\phi^*d\theta=\phi^*(r^2d\theta)=|u_{\epsilon}|^{-2}ju_{\epsilon},$$
so that applying (\ref{intdivfree}) to (\ref{exact}) yields
$$\int_M|u_{\epsilon}|^2|d\psi|^2=0.$$
Thus, $|d\phi|=|\phi^*d\theta|=0,$ so that $\phi\equiv \beta$ for some constant $\beta\in S^1$.

\hspace{6mm} It then follows from (\ref{eleqn}) that
$$\Delta(1-|u_{\epsilon}|)=\epsilon^{-2}(1-|u_{\epsilon}|^2)|u_{\epsilon}|.$$
Multiplying both sides by $(1-|u_{\epsilon}|)$ and integrating yields
$$0\geq-\int_M|d(1-|u_{\epsilon}|)|^2=\epsilon^{-2}\int_M(1-|u_{\epsilon}|)^2|u_{\epsilon}|(1+|u_{\epsilon}|)\geq 0,$$
and we immediately conclude that $|u_{\epsilon}|\equiv 1;$ hence, $u_{\epsilon}\equiv \beta\in S^1$ is a trivial solution.

\hspace{6mm} Thus, if $u_{\epsilon}$ is a nontrivial solution of (\ref{eleqn}) on $M$ with $H^1_{dR}(M)=0$, there must be some point $x_{\epsilon} \in M$ such that $u_{\epsilon}(x_{\epsilon})=0$. Now we appeal to one of the central analytical lemmas of \cite{BBO} (see also \cite{LR})--the so-called $\eta$-ellipticity theorem--to see that the existence of such a zero necessarily produces the desired energy blow up. Though the $\eta$-ellipticity theorem is originally stated for the Euclidean setting in \cite{BBO}, the arguments are purely local, and can be applied to small balls on compact manifolds to yield the following

\begin{thm}\label{etaellip}\emph{(Theorem 2 of \cite{BBO})} There exist positive constants $\epsilon_0(M),\delta_0(M),\eta_0(M)>0$ such that if $u_{\epsilon}$ solves $(\ref{eleqn})$ on a geodesic ball $B_r(x)$, where $\epsilon<\epsilon_0$, $r\leq\delta_0$ and
\begin{equation}\label{lowenergy}
\int_{B_r(x)}e_{\epsilon}(u_{\epsilon})\leq r^{n-2}\eta_0 |\log (\epsilon/r)|,
\end{equation}
then\footnote{Our choice of the constant $\frac{7}{8}$ here was of course somewhat arbitrary; we could replace it with any constant in $(0,1)$, changing $\eta_0$ accordingly.}
\begin{equation}
|u_{\epsilon}|^2(x)\geq \frac{7}{8}.
\end{equation}
\end{thm}

Applying this at the zeros $x_{\epsilon}$ of our nontrivial solutions $u_{\epsilon}$, with $r=\delta_0(M)$, we see that for all $\epsilon$ sufficiently small, we must have
\begin{equation}
E_{\epsilon}(u_{\epsilon})\geq \int_{B_{\delta_0}(x_{\epsilon})}e_{\epsilon}(u_{\epsilon})>\delta_0^{n-2}\eta_0(|\log\epsilon|-|\log\delta_0|),
\end{equation}
from which (\ref{goodbound}) follows.
\end{proof}

\hspace{6mm} Applying the preceding lemma to the nontrivial solutions of Proposition \ref{existence}, we immediately obtain

\begin{lem}\label{scminmax} If $M^n$ is a Riemannian manifold with $H^1_{dR}(M)=0,$ then the min-max constants $c_{\epsilon}(M)$ defined by $(\ref{cdef})$ satisfy the lower bound of $(\ref{cbounds})$: namely,
\begin{equation}\label{scbounds}
\liminf_{\epsilon\to 0}\frac{c_{\epsilon}(M)}{|\log\epsilon|}>0.
\end{equation}
\end{lem}

\hspace{6mm} Next, we observe that these lower bounds can be extended to arbitrary manifolds by way of a simple trick, which can easily be applied to a wide range of min-max constructions. The resulting energy estimates are somewhat crude, but sufficient to establish the desired energy blow-up.

\hspace{6mm} Let $(M^n,g)$ once again be an arbitrary compact manifold, and recall the definition of $\Gamma(M)$:
$$\Gamma(M):=\{F\in C^0(D, H^1(M,\mathbb{R}^2))\mid F(y)\equiv y\text{ for }y\in S^1\}.$$
Given a domain $\Omega\subset M$ and a family $F\in \Gamma(M)$, it's clear that the family $F|_{\Omega}\in C^0(D, H^1(\Omega,\mathbb{R}^2))$ given by restriction
$$y\mapsto F(y)|_{\Omega}$$
lies in $\Gamma(\Omega)$, and trivially satisfies the bound
$$E_{\epsilon}(F(y))\geq E_{\epsilon}(F(y)|_{\Omega}).$$
As a consequence, we obtain the simple estimate
\begin{equation}\label{restriction}
c_{\epsilon}(M)\geq \inf_{F\in \Gamma(\Omega)}\max_{y\in D}E_{\epsilon}(F(y))
\end{equation}
for any subdomain $\Omega\subset M$.

\hspace{6mm} Now, let $B^n\subset M$ be an embedding of the closed $n$-ball into $M$ (e.g., as a closed geodesic ball), and consider the map 
$$R: H^1(B^n,\mathbb{R}^2)\to H^1(S^n,\mathbb{R}^2)$$
given by identifying $B^n$ with a closed hemisphere and reflecting. That is, for $u\in H^1(B^n,\mathbb{R}^2)$, define
$$Ru(x_0,\ldots,x_n):=u\circ f(|x_0|,x_1,\ldots,x_n),$$
where $f: S^n_+\to B^n$ is a diffeomorphism with the closed hemisphere 
$$S^n_+=\{(x_0,\ldots,x_n)\in S^n\mid x_0\geq 0\}.$$
It's then straightforward to check that $R$ is a bounded (hence continuous) linear map, and in particular,
\begin{equation}\label{rinfam}
R\circ F\in \Gamma(S^n)\text{ for any }F\in \Gamma(B^n).
\end{equation}
Moreoever, since reflection across the equator simply doubles the energy $E_{\epsilon}$ of a map in $H^1(S^n_+,\mathbb{R}^2)$, and $f: S^n_+\to B^n$ is necessarily bi-Lipschitz, we have an estimate of the form 
\begin{equation}
C^{-1}E_{\epsilon}(Ru)\leq E_{\epsilon}(u)\leq CE_{\epsilon}(Ru)\text{ for every }u\in H^1(M,\mathbb{R}^2),
\end{equation}
for some constant $C$ depending on our choice of $f$.

\hspace{6mm} Applying (\ref{restriction}) to a fixed choice of closed ball $B^n\subset M$, and fixing a choice of $f: S^n_+\to B^n$, we conclude from (\ref{rinfam}) that
\begin{equation}\label{spherebound}
c_{\epsilon}(M)\geq C^{-1}c_{\epsilon}(S^n,g_{standard})
\end{equation}
for some finite, positive constant $C$ independent of $\epsilon$. Finally, we note that since $H^1_{dR}(S^n)=0$, we can apply Lemma \ref{scminmax} to $(S^n,g_{standard})$, and combining (\ref{scbounds}) with (\ref{spherebound}), we arrive at the desired lower bound:

\begin{prop}\label{lbds} On any compact manifold $(M^n,g)$, the min-max constants $c_{\epsilon}(M)$ satisfy
\begin{equation}
\liminf_{\epsilon\to 0}\frac{c_{\epsilon}(M)}{|\log\epsilon|}>0.
\end{equation}
\end{prop}

\section{Upper Bounds on the Energies}

\hspace{6mm} To find suitable upper bounds for the energies $c_{\epsilon}(M)$, we just need to produce families $F_{\epsilon}\in \Gamma(M)$ consisting of maps that behave roughly like model solutions of (\ref{eleqn}).

\hspace{6mm} Given $\epsilon>0$, consider the map $v_{\epsilon}:\mathbb{R}^2\to \mathbb{R}^2$ defined by
\begin{equation}\label{vedef}
v_{\epsilon}(z)=\frac{z}{|z|},\text{ for } |z|>\epsilon,\text{ }v_{\epsilon}(z)=\frac{z}{\epsilon}\text{ for }|z|\leq \epsilon.
\end{equation}
Letting $\pi_z^{\perp}:\mathbb{R}^2\to \mathbb{R}^2$ denote orthogonal projection onto $[\mathbb{R}z]^{\perp}$, we then have
\begin{equation}\label{dve}
dv_{\epsilon}(z)=\frac{\pi_z^{\perp}}{|z|}\text{ for }|z|>\epsilon\text{ and }dv_{\epsilon}(z)=\frac{1}{\epsilon}Id\text{ for }|z|\leq \epsilon,
\end{equation}
and, in particular,
\begin{equation}\label{vest1}
e_{\epsilon}(v_{\epsilon})=\frac{1}{2}|dv_{\epsilon}|^2(z)+\frac{W(v_{\epsilon}(z))}{\epsilon^2}\leq \frac{1}{2|z|^2}\text{ for }|z|>\epsilon,\text{ and }\leq \frac{9}{4\epsilon^2}\text{ for }|z|\leq \epsilon.
\end{equation}
A quick computation then reveals an energy bound of the form
\begin{equation}\label{diskestimate}
E_{\epsilon}(v_{\epsilon},D_R)=\int_{\{|z|\leq R\}}e_{\epsilon}(v_{\epsilon})\leq \pi\log(R/\epsilon)+C
\end{equation}
for the restriction of $v_{\epsilon}$ to the disk $D_R$ of radius $R$ about the origin. 

\hspace{6mm} Let $\Omega \subset \mathbb{R}^2$ be a bounded domain, and consider the family of maps 
$$D\ni y \mapsto v_{y,\epsilon}\in Lip(\Omega,\mathbb{R}^2)$$
given by the translates
\begin{equation}\label{vydef}
v_{y,\epsilon}(z)=v_{\epsilon}(z+\frac{y}{1-|y|})\text{ for }|y|<1,
\end{equation}
and
\begin{equation}\label{vybdry}
v_{y,\epsilon}(z)=y\text{ for }y\in \partial D.
\end{equation}
Since $\Omega$ is bounded, it follows from (\ref{vedef}) and (\ref{dve}) that $y\mapsto v_{y,\epsilon}$ is a continuous family in $Lip(\Omega,\mathbb{R}^2)$, and thus, by (\ref{vybdry}), a member of $\Gamma(\Omega)$. In light of the energy estimate (\ref{diskestimate}), this family seems like a promising starting point for constructing well-behaved families on an arbitrary manifold.

\hspace{6mm} Now, let $M$ be a compact manifold, and let $f: M\to \mathbb{R}^2$ be a Lipschitz map. By the preceding discussion, it's clear that 
\begin{equation}\label{famdef}
D\ni y\mapsto F_y:=v_{y,\epsilon}\circ f
\end{equation}
defines a valid family in $\Gamma(M)$; thus, we can estimate the min-max constants $c_{\epsilon}$ from above by making a reasonable choice of $f\in Lip(M,\mathbb{R}^2)$. 

With $f$ and $F_y$ as above, setting $w:=\frac{-y}{1-|y|}$, it follows from (\ref{vest1}) that
\begin{equation}\label{fest1}
E_{\epsilon}(F_y)\leq\int_{f^{-1}(\mathbb{C}\setminus D_{\epsilon}(w))} \frac{1}{2}Lip(f)^2\frac{1}{2|f(x)-w|^2}+\frac{9}{4\epsilon^2}|f^{-1}(D_{\epsilon}(w))|,
\end{equation}
Suppose now that the Jacobian $|Jf|=|df^1\wedge df^2|$ and the level sets $f^{-1}(\{z\})$ of $f$ satisfy estimates of the form
\begin{equation}\label{fcond1}
|Jf(x)|\geq C^{-1} \text{ a.e. }x\in M
\end{equation}
and
\begin{equation}\label{fcond2}
\sup_{z\in \mathbb{C}}\mathcal{H}^{n-2}(f^{-1}(\{z\}))\leq C
\end{equation}
for some finite, positive constants $C$.
Then the coarea formula for Lipschitz maps (as stated in, e.g., \cite{EG}), together with (\ref{fest1}), yields 
\begin{eqnarray*}
E_{\epsilon}(F_y)&\leq &CLip(f)^2\int_{f(M)\setminus D_{\epsilon}(w)}\frac{1}{|z-w|^2}\cdot \mathcal{H}^{n-2}(f^{-1}(\{z\}))dz\\
&&+\frac{C}{\epsilon^2}\int_{D_{\epsilon}(w)}\mathcal{H}^{n-2}(f^{-1}\{z\})dz\\
&\leq &C^2\int_{f(M)\setminus D_{\epsilon}(w)}\frac{1}{|z-w|^2}+C^2.
\end{eqnarray*}
Finally, since the image $f(M)$ is a bounded subset of $\mathbb{R}^2$, we arrive at an estimate
\begin{equation}\label{goodest}
E_{\epsilon}(F_y)\leq C_1|\log\epsilon|+C_2,
\end{equation}
where $C_1$ and $C_2$ are constants depending only on $f$. Summarizing, we've proved the following:

\begin{lem}\label{flemma} Given a Lipschitz map $f: M\to \mathbb{R}^2$ satisfying estimates of the form $(\ref{fcond1})$ and $(\ref{fcond2})$, the families $F^{\epsilon}\in \Gamma(M)$ defined by
$$F^{\epsilon}(y):=v_{y,\epsilon}\circ f$$
satisfy
\begin{equation}\label{febounds}
\limsup_{\epsilon\to 0}\frac{1}{|\log\epsilon|}\max_{y\in D}E_{\epsilon}(F^{\epsilon}(y))<\infty.
\end{equation}
\end{lem}

\hspace{6mm} Our goal now is to construct $f \in Lip(M,\mathbb{R}^2)$ satisfying (\ref{fcond1}) and (\ref{fcond2}). We do this via triangulation. Let $\Phi: M\to |\mathcal{K}|$ be a bi-Lipschitz map from $M$ to the underlying space of a finite simplicial complex $\mathcal{K}$ in some $\mathbb{R}^L$ (see, e.g., \cite{Wh} for the classical construction). For each $k$-simplex $\Delta \in \mathcal{K}$, denote by $V(\Delta)$ the $k$-plane through the origin of $\mathbb{R}^L$ parallel to $\Delta$. Since $\mathcal{K}$ is finite, we can choose a generic $2$-plane $\Pi \subset \mathbb{R}^L$ such that the restriction 
$$p|_{V(\Delta)}:V(\Delta)\to \Pi$$
of the orthogonal projection
$$p:\mathbb{R}^L\to \Pi$$
has rank $2$ for every $\Delta \in \mathcal{K}$ of dimension $\geq 2$ and rank $1$ when $\dim \Delta=1$. Now identify $\Pi$ with $\mathbb{R}^2$, and set
\begin{equation}\label{fdef}
f:=p\circ \Phi.
\end{equation}

\hspace{6mm} Since $\Phi$ is bi-Lipschitz, $\exists$ $c>0$ such that, for a.e. $x \in M$, the pullback 
$$\Phi^*: \bigwedge^2T_{\Phi(p)}^*|\mathcal{K}|\to \bigwedge^2T_p^*M$$
satisfies
\begin{equation}\label{phistarest}
|\Phi^*(\zeta)|\geq c|\zeta|\text{ for every }\zeta\in \bigwedge^2T_{\Phi(p)}^*|\mathcal{K}|.
\end{equation}
Furthermore, almost every $x\in M$ lies in the preimage of the interior $\Delta^{\circ}$ of some $n$-dimensional simplex $\Delta \in \mathcal{K}$. At such a point $x$, the differential $df$ of (\ref{fdef}) is given by
$$p|_{V(\Delta)}\circ d\Phi,$$
and since $p|_{V(\Delta)}$ has full rank by our choice of $\Pi$, it follows that
\begin{equation}
|Jf(x)|=|d\Phi^*(p_{V(\Delta)}^*(e^1\wedge e^2))|\geq c|p_{V(\Delta)}^*(e^1\wedge e^2)|\geq C^{-1}
\end{equation}
for some finite positive constant $C$. Thus, our chosen $f$ satisfies (\ref{fcond1}), and it remains to check (\ref{fcond2}).

\hspace{6mm} This is similarly straightforward. For each $\Delta \in \mathcal{K}$ and $z\in \Pi$, our constraints on the rank of $p|_{V(\Delta)}$ imply that $p^{-1}(\{z\})\cap \Delta$ is given by the intersection of $\Delta$ with a translate of some subspace of $V(\Delta)$ of dimension $\leq n-2$. Consequently, we have simple bounds of the form
$$\mathcal{H}^{n-2}(p^{-1}(\{z\}\cap \Delta)\leq c_n\cdot diam(\Delta)^{n-2},$$
and thus, letting $N$ denote the number of simplices in $\mathcal{K}$,
\begin{equation}
\sup_{z\in \Pi}\mathcal{H}^{n-2}(p^{-1}(\{z\})\cap |\mathcal{K}|)\leq N c_n diam(\Delta)^{n-2}<\infty.
\end{equation}
It then follows that
\begin{equation}
\mathcal{H}^{n-2}(f^{-1}(\{z\}))\leq Lip(\Phi^{-1})^{n-2}\mathcal{H}^{n-2}(p^{-1}(\{z\})\cap|\mathcal{K}|)\leq C,
\end{equation}
for each $z\in \Pi$, so that (\ref{fcond2}) holds as well.

\hspace{6mm} Thus, $f$ satisfies the hypotheses of Lemma \ref{flemma}, so we have families $F^{\epsilon}\in \Gamma(M)$ satisfying the estimate (\ref{febounds}), and consequently,
\begin{equation}\label{upperbds}
\lim\sup_{\epsilon\to 0}\frac{c_{\epsilon}(M)}{|\log\epsilon|}=\limsup_{\epsilon\to 0}\frac{1}{|\log\epsilon|}\inf_{F\in \Gamma(M)}\max_{y\in D}E_{\epsilon}(F(y))<\infty.
\end{equation}
Proposition \ref{lbds} and (\ref{upperbds}) then combine to give us the estimate (\ref{cbounds}), completing the proof of Theorem \ref{mainthm}.

\section{The Energy Concentration Varifold when $H^1_{dR}(M)=0$}

\hspace{6mm} For $\epsilon \in (0,1)$, let $u_{\epsilon}$ be a solution of (\ref{eleqn}) as constructed in Theorem \ref{mainthm}. Our goal in this section is to show that when $H^1_{dR}(M)=0$, along some subsequence $\epsilon_j\to 0$, the energy concentration measures
$$\mu_{\epsilon}:=\frac{e_{\epsilon}(u_{\epsilon})}{|\log\epsilon|}dv_g$$
concentrate on a stationary, rectifiable $(n-2)$-varifold. The key analytical lemma we'll need to establish this is the density estimate
\begin{lem}\label{denslem} If $H^1_{dR}(M)=0$, then for any limiting measure $\mu=\lim_{j\to\infty} \mu_{\epsilon_j}$, the $(n-2)$-density 
$$\Theta^{n-2}(\mu,x):=\lim_{r\to 0}\frac{\mu(B_r(x))}{\omega_{n-2}r^{n-2}}>0$$
for every $x\in spt(\mu).$
\end{lem}

\hspace{6mm} To see why Lemma (\ref{denslem}) is sufficient to establish Thorem \ref{thm2}, we recall some of the results of Ambrosio and Soner in \cite{AS}. Given an integer $0\leq m\leq n$, denote by 
$$A_m(M)\subset End(TM)$$
the compact (fiber-)subbundle of $End(TM)$ given by\footnote{In the definition of generalized varifold in \cite{AS}, the trace inequality $tr(S)\geq m$ is replaced by equality, but the upper bound on the trace plays no role in their analysis.}
\begin{equation}
A_m(M):=\{S \in End(TM)\mid S=S^*,-nId\leq S\leq Id,\text{ }tr(S)\geq m.\}
\end{equation}
In the language of \cite{AS}, a \emph{generalized $m$-varifold} is a nonnegative Radon measure on the fiber bundle $A_m(M)$. Note that the Grassmannian bundle $G_m(TM)$ is naturally included in $A_m(M)$ by identifying $m$-dimensional subspaces with the associated orthogonal projections, and thus every standard $m$-varifold (in the sense of \cite{All},\cite{Si}) also defines a generalized $m$-varifold.

\hspace{6mm} As with standard $m$-varifolds, one can define the first variation $\delta V$ of a generalized $m$-varifold $V$ as follows \cite{AS}: given a smooth vector field $X$ on $M$, we set
\begin{equation}\label{dvdef}
\delta V(X):=\int_{A_m(M)}\langle S,\nabla X\rangle dV(S),
\end{equation}
and we call $V$ \emph{stationary} if $\delta V=0$. Naturally, one also defines the mass measure $\|V\|$ on $M$ as the pushforward of $V$ by the projection $\pi: A_m(M)\to M$ \cite{AS}. Given a sequence $V_j$ of generalized $m$-varifolds converging in $C^0(A_m)^*$ to a generalized $m$-varifold $V$, it follows from the standard properties of nonnegative Radon measures on compact spaces that $\|V_j\|\to \|V\|$ in $C^0(M)^*$; moreover, if $\delta V_j=0$, it follows immediately from the definition of $(C^0)^*$ convergence that $\delta V=0$ as well \cite{AS}.

\hspace{6mm} Aside from standard varifolds, the most important (and motivating) examples of generalized varifolds come from the stress-energy tensors associated with solutions of pdes satisfying an inner variation equation. For solutions of (\ref{eleqn}), one considers the tensor (viewed as a symmetric endomorphism)
\begin{equation}\label{tdef}
T_{\epsilon}(u_{\epsilon}):=e_{\epsilon}(u)Id-du^*du;
\end{equation}
it follows from (\ref{eleqn}) that 
$$div(T_{\epsilon})=0,$$
and thus, for any smooth vector field $X$ on $M$,
\begin{equation}\label{iveqn}
\int_M\langle T_{\epsilon}(u_{\epsilon}),\nabla X\rangle=0.
\end{equation}
In particular, writing 
\begin{equation}\label{pedef}
P_{\epsilon}:=Id-e_{\epsilon}(u_{\epsilon})^{-1}du^*du,
\end{equation}
it follows that the measure
\begin{equation}\label{vedef}
V_{\epsilon}:=\delta_{P_{\epsilon}}\times \mu_{\epsilon}
\end{equation}
defines a stationary generalized $(n-2)$-varifold with weight measure $\mu_{\epsilon}$.

\hspace{6mm} For solutions, like those of Theorem \ref{mainthm}, satisfying an energy bound $\mu_{\epsilon}(M)\leq C$, we can extract a subsequence $\epsilon_j\to 0$ such that
$$V_{\epsilon_j}\to V$$
as generalized varifolds, so that $V$ is again a stationary generalized $(n-2)$-varifold with weight measure
$$\|V\|=\mu:=\lim_{\epsilon_j\to 0}\mu_{\epsilon_j}.$$
We now recall the key measure-theoretic result of \cite{AS}:

\begin{prop}\label{asprop} (Theorem 3.8 of \cite{AS}) If $V$ is a generalized $m$-varifold for which $\delta V\in [C^0(M,TM)]^*$ and $\Theta^m(\|V\|,x)>0$ at $\|V\|$-a.e. $x\in spt(V)$, then there is a rectifiable $m$-varifold $\tilde{V}$ with $\|\tilde{V}\|=\|V\|$ and $\delta \tilde{V}=\delta V$.
\end{prop}

\begin{remark} If the given $V$ is a (standard) varifold, this is simply Allard's rectifiability theorem (see \cite{All}, Section 5). The key observation of \cite{AS} is that positive density and bounded first variation force the fiber-wise center of mass of $V$ to have the structure of orthogonal projection onto an $m$-dimensional subspace. In particular, if $V$ has the structure of a Dirac mass in each fiber of $A_m$ (as in (\ref{vedef})), the induced varifold $\tilde{V}=V$.
\end{remark}

\begin{remark} In our statements of Lemma \ref{denslem} and Proposition \ref{asprop}, we have implicitly used the fact that generalized $m$-varifolds $V$ with bounded first variation satisfy a monotonicity property identical to that of standard $m$-varifolds, to ensure that the density $\Theta^m(\|V\|,x)$ is well-defined without the decorations $^*$ or $_*$. (The monotonicity follows from the standard computations of, e.g., Section 5 of \cite{All}, or Section 40 of \cite{Si}.)
\end{remark}

From Proposition \ref{asprop} and the preceding discussion, it is now clear that Lemma \ref{denslem} will be sufficient to establish the conclusion of Theorem \ref{thm2}; the remainder of this section will be devoted to establishing this positive density condition.

\hspace{6mm} To this end, we need a better understanding of the structure of the limiting measure $\mu$. By the local estimates of (\cite{BOS}, Theorem 1.1 and Proposition 1.3), the upper bound in (\ref{cbounds}) immediately gives us the uniform bound
\begin{equation}\label{wetcbounds}
\int_M|d|u_{\epsilon}||^2+\frac{W(u_{\epsilon})}{\epsilon^2}\leq C,
\end{equation}
where $C$ is independent of $\epsilon$. Defining $ju_{\epsilon}$ as before (\ref{jdef}), we observe that 
\begin{equation}\label{ududecomp}
|u_{\epsilon}|^2|du_{\epsilon}|^2=|ju_{\epsilon}|^2+|u_{\epsilon}|^2|d|u_{\epsilon}||^2,
\end{equation}
so it suffices to understand the contributions of $|ju_{\epsilon}|^2$ and $(1-|u_{\epsilon}|^2)|du_{\epsilon}|^2$ to the energy concentration.

\hspace{6mm} Next, note that the equation (\ref{eleqn}) on $(M,g)$ is equivalent to 
\begin{equation}\label{rescale}
\Delta_{g_{\epsilon}}u_{\epsilon}=DW(u_{\epsilon})=-(1-|u_{\epsilon}|^2)u_{\epsilon}
\end{equation}
in the dilated metric $g_{\epsilon}=\epsilon^{-2}g$. It follows that
$$\Delta_{g_{\epsilon}}\frac{1}{2}(1-|u_{\epsilon}|^2)=(1-|u_{\epsilon}|^2)|u_{\epsilon}|^2-|du_{\epsilon}|_{g_{\epsilon}}^2$$
and, via the Bochner formula,
$$\Delta_{g_{\epsilon}}\frac{1}{2}|du_{\epsilon}|_{g_{\epsilon}}^2=\frac{1}{2}|d|u_{\epsilon}|^2|_{g_{\epsilon}}^2-(1-|u_{\epsilon}|^2)|du_{\epsilon}|_{g_{\epsilon}}^2+\langle Ric_{g_{\epsilon}},du_{\epsilon}^*du_{\epsilon}\rangle_{g_{\epsilon}}+|Hess(u_{\epsilon})|_{g_{\epsilon}}^2.$$
Thus, setting
$$w:=\frac{1}{2}|du_{\epsilon}|_{g_{\epsilon}}^2-\frac{b}{2}(1-|u_{\epsilon}|^2),$$
for some constant $b$, we find that
\begin{eqnarray*}
\Delta_{g_{\epsilon}}w&\geq& (b-1)|du_{\epsilon}|_{g_{\epsilon}}^2+2|u_{\epsilon}|^2w+\epsilon^4\langle Ric_g,du_{\epsilon}^*du_{\epsilon}\rangle_g\\
&\geq &(b-1-\epsilon^2|Ric_g^-|_g)|du_{\epsilon}|_{g_{\epsilon}}^2+2|u_{\epsilon}|^2w.
\end{eqnarray*}
Writing $A(M,g):=\max_{x\in M}|Ric_g^-|_g$, it follows that for $b>1+A\epsilon^2$, $w$ must be negative at its maximum, and we therefore conclude that
$$|du_{\epsilon}|_{g_{\epsilon}}^2\leq (1+A\epsilon^2)(1-|u_{\epsilon}|^2)$$
everywhere. Scaling back, we arrive at the gradient estimate
\begin{equation}\label{bochest}
|du_{\epsilon}|_g^2\leq (\frac{1}{\epsilon^2}+A)(1-|u_{\epsilon}|^2).
\end{equation}

From (\ref{bochest}) and (\ref{wetcbounds}), we obtain the estimate
\begin{equation}\label{sigmadu}
\int_M(1-|u_{\epsilon}|^2)|du_{\epsilon}|^2\leq C'\int_M\frac{W(u_{\epsilon})}{\epsilon^2}\leq C'',
\end{equation}
so that the only nontrivial contribution to the energy blow-up must come from $ju_{\epsilon}$. That is, given a convergent subsequence $\mu_{\epsilon_j}\to\mu$, combining (\ref{wetcbounds}), (\ref{ududecomp}), and (\ref{sigmadu}), we conclude that 
\begin{equation}\label{jdominates}
\mu=\lim_{j\to \infty}\frac{1}{2}\frac{|ju_{\epsilon_j}|^2}{|\log\epsilon_j|}dv_g.
\end{equation}

\hspace{6mm} Now, following the arguments of \cite{BBO}, choose a smooth function $f: [0,1]\to [1,2]$ satisfying
\begin{equation}\label{fdef}
f(t)=\frac{1}{t}\text{ for }t\geq \frac{3}{4},\text{ }f(t)=1\text{ for }t\leq \frac{1}{2},\text{ and }|f'|\leq 2.
\end{equation}
Defining the one-forms
\begin{equation}\label{gamdef}
\gamma_{\epsilon}:=f(|u_{\epsilon}|^2)ju_{\epsilon},
\end{equation}
it follows from the choice of $f$ that
$$||\gamma_{\epsilon}|^2-|ju_{\epsilon}|^2|\leq C(1-|u_{\epsilon}|^2)|ju_{\epsilon}|^2\leq C(1-|u_{\epsilon}|^2)|du_{\epsilon}|^2,$$
and thus, in light of (\ref{jdominates}) and (\ref{sigmadu}), we have
\begin{equation}\label{gammadom}
\mu=\lim_{j\to\infty}\frac{1}{2}\frac{|\gamma_{\epsilon_j}|^2dv_g}{|\log\epsilon_j|}.
\end{equation}

\hspace{6mm} As in \cite{BBO}, our estimates for $\gamma_{\epsilon_j}$ will come from estimates on the components of its Hodge decomposition; naturally, this is the point in our analysis where the constraint $H^1_{dR}(M)=0$ becomes crucial. (We also assume throughout that $M$ is orientable, but this is ultimately of no analytic significance, as we can always pass to a double cover.) Choosing $\theta_{\epsilon}\in C^{\infty}(M)$ and $\xi_{\epsilon}\in \Omega^2(M)$ such that 
\begin{equation}\label{thetadef}
\Delta\theta_{\epsilon}=div(\gamma_{\epsilon})
\end{equation}
and 
\begin{equation}\label{xidef}
\Delta_H\xi_{\epsilon}=d\gamma_{\epsilon}
\end{equation}
(where $\Delta_H=dd^*+d^*d$ is the usual Hodge Laplacian), we use the fact that $H^1_{dR}(M)=0$ to conclude that
\begin{equation}\label{hodgedecomp}
\gamma_{\epsilon}=d\theta_{\epsilon}+d^*\xi_{\epsilon}.
\end{equation}

\begin{remark} Without the assumption that $H^1_{dR}(M)=0$, we could still carry out the analysis of this section provided we had some a priori control on the harmonic part of $\gamma_{\epsilon}$. However, without such control, we allow for solutions like those discussed in Remark \ref{harmex}, for which the harmonic part of $\gamma_{\epsilon}$ dominates the energy blow-up.
\end{remark}

\hspace{6mm} We show next that the exact part $d\theta_{\epsilon}$ of $\gamma_{\epsilon}$ contributes negligibly to $\mu$. Since $div(ju_{\epsilon})=0$, the defining equation (\ref{thetadef}) for $d\theta_{\epsilon}$ and (\ref{gamdef}) yield
$$\Delta \theta_{\epsilon}=div(\gamma_{\epsilon})=f'(|u_{\epsilon}|^2)\langle d|u_{\epsilon}|^2,ju_{\epsilon}\rangle.$$
Multiplying by $\theta_{\epsilon}$ and integrating, we see that the $L^2$ norm of $d\theta_{\epsilon}$ is given by
\begin{equation}
\int_M |d\theta_{\epsilon}|^2=-\int_M\theta_{\epsilon}f'(|u_{\epsilon}|^2)\langle d|u_{\epsilon}|^2,ju_{\epsilon}\rangle.
\end{equation}
Applying the coarea formula to the $|u_{\epsilon}|^2$ terms, we recast this as
\begin{eqnarray*}
\int_M|d\theta_{\epsilon}|^2&=&-\int_0^1f'(t)\left(\int_{\partial \{|u_{\epsilon}|^2<t\}}\langle \theta_{\epsilon}ju_{\epsilon},\nu\rangle\right)dt\\
&=&-\int_0^1f'(t)\left(\int_{\{|u_{\epsilon}|^2<t\}}div(\theta_{\epsilon}ju_{\epsilon})\right)dt\\
&=&-\int_0^1f'(t)\left(\int_{\{|u_{\epsilon}|^2<t\}}\langle d\theta_{\epsilon},ju_{\epsilon}\rangle\right)dt\\
&\leq &2\int_0^1\left(\int_{\{|u_{\epsilon}|^2<t\}}|d\theta_{\epsilon}||ju_{\epsilon}|\right)dt\\
&\leq &\int_0^1\left(\frac{1}{2}\int_{\{|u_{\epsilon}|^2<t\}}|d\theta_{\epsilon}|^2+2|ju_{\epsilon}|^2\right)dt,
\end{eqnarray*}
from which it follows that
\begin{equation}\label{coareacons}
\int_M|d\theta_{\epsilon}|^2\leq 4\int_0^1\left(\int_{\{|u_{\epsilon}|^2<t\}}|ju_{\epsilon}|^2\right)dt.
\end{equation}
Now, by (\ref{bochest}), we know that 
\begin{equation}\label{jbochbds}
|ju_{\epsilon}|^2\leq |du_{\epsilon}|^2\leq \frac{C}{\epsilon^2},
\end{equation}
while it follows from (\ref{wetcbounds}) that
\begin{equation}\label{volbounds}
|\{|u_{\epsilon}|^2\leq t\}|\leq \frac{4\epsilon^2}{(1-t)^2}\int_M\frac{W(u_{\epsilon})}{\epsilon^2}\leq \frac{C\epsilon^2}{(1-t)^2}.
\end{equation}
Splitting the right-hand side of (\ref{coareacons}) into integrals over $t\in [0,1-|\log\epsilon|^{-1/2}]$ and $t\in [1-|\log\epsilon|^{-1/2},1]$ (taking now $\epsilon<\frac{1}{e}$), we apply (\ref{jbochbds}), (\ref{volbounds}), and (\ref{cbounds}) to estimate
\begin{eqnarray*}
\int_M|d\theta_{\epsilon}|^2&\leq &4\int_0^{1-|\log\epsilon|^{-1/2}}\left(\int_{\{|u_{\epsilon}|^2<t\}}|ju_{\epsilon}|^2\right)dt+\int_{1-|\log\epsilon|^{-1/2}}^1\left(\int_{\{|u_{\epsilon}|^2<t\}}|ju_{\epsilon}|^2\right)dt\\
&\leq &4\int_0^{1-|\log\epsilon|^{-1/2}}\frac{C}{\epsilon^2}\cdot\frac{C\epsilon^2}{(1-t)^2}dt+|\log\epsilon|^{1/2}\mu_{\epsilon}(M)\\
&\leq &C'|\log\epsilon|^{1/2}.
\end{eqnarray*}
It then follows that, for any $U\subset M$,
\begin{eqnarray*}
\frac{1}{|\log\epsilon|}\int_U(|d\theta_{\epsilon}|^2+2|\langle d\theta_{\epsilon},d^*\xi_{\epsilon}\rangle|)&\leq &\frac{C}{|\log\epsilon|^{1/2}}+\frac{1}{|\log\epsilon|}\int_U |\log\epsilon|^{1/4}|d\theta_{\epsilon}|^2\\
&&+\frac{1}{|\log\epsilon|}\int_U|\log\epsilon|^{-1/4}|d^*\xi_{\epsilon}|^2\\
&\leq &\frac{C}{|\log\epsilon|^{1/4}}\\
&\to & 0\text{ as }\epsilon\to 0,
\end{eqnarray*}
and consequently, we obtain our final reduction
\begin{equation}\label{xired}
\mu=\lim_{j\to\infty}\frac{1}{2}\frac{|d^*\xi_{\epsilon_j}|^2dv_g}{|\log\epsilon_j|}.
\end{equation}

\hspace{6mm} Now, consider a point $x \in M$ at which the density $\Theta^{n-2}(\mu,x)=0$; that is, suppose
\begin{equation}\label{densvan}
\lim_{r\to 0}\frac{\mu(B_r(x))}{r^{n-2}}=0.
\end{equation}
To establish Lemma \ref{denslem}, we need to show that $x \notin spt(\mu)$; i.e., we need to find some ball $B_r(x)$ about $x$ for which
\begin{equation}\label{muvan}
\mu(B_r(x))=\lim_{j\to \infty}\frac{1}{2|\log\epsilon_j|}\int_{B_r(x)}|d^*\xi_{\epsilon_j}|^2=0.
\end{equation}

\hspace{6mm} We begin by arguing as in Section VII of \cite{BBO}. Let $\delta_0(M),\eta_0(M)>0$ be as in the $\eta$-ellipticity Theorem \ref{etaellip}; by (\ref{densvan}), we can select $R\in (0,\delta_0)$ such that 
$$\frac{1}{2}\eta_0\geq \frac{\mu(B_{2R}(x))}{R^{n-2}}=\lim_{\epsilon\to 0}\frac{R^{2-n}}{|\log\epsilon|}\int_{B_{2R}(x)}e_{\epsilon}(u_{\epsilon})=\lim_{\epsilon\to 0}\frac{R^{2-n}}{|\log (\frac{\epsilon}{R})|}\int_{B_{2R}(x)}e_{\epsilon}(u_{\epsilon}).$$
Applying Theorem \ref{etaellip} at each point in $B_R(x)$ for $\epsilon$ sufficiently small, we conclude that
\begin{equation}\label{ulowbound}
|u_{\epsilon}|^2(y)\geq \frac{7}{8}\text{ for every }y\in B_R(x),
\end{equation}
and observe that, by (\ref{fdef}) and the definition of $\gamma_{\epsilon}$, 
$$\gamma_{\epsilon}=\frac{1}{|u_{\epsilon}|^2}ju_{\epsilon}=j(u_{\epsilon}/|u_{\epsilon}|)\text{ on }B_R(x).$$
In particular, it follows that
\begin{equation}
dd^*\xi_{\epsilon}=d\gamma_{\epsilon}=0\text{ on }B_R(x),
\end{equation}
and defining $\varphi_{\epsilon}\in C^{\infty}(B_R(x))$ by
\begin{equation}\label{varphidef}
\Delta \varphi_{\epsilon}=0,\text{ }\frac{\partial \varphi_{\epsilon}}{\partial \nu}=d^*\xi_{\epsilon}(\nu)\text{ on }\partial B_R(x),\text{ and }\int_{B_R(x)} \varphi_{\epsilon}=0, 
\end{equation}
we have
\begin{equation}
d\varphi_{\epsilon}=d^*\xi_{\epsilon}\text{ on }B_R(x).
\end{equation}
Since $\Delta \varphi_{\epsilon}=0$ and $\int_{B_R(x)}\varphi_{\epsilon}=0$, it follows from standard elliptic estimates that
\begin{equation}\label{ellipregest}
\int_{B_{R/2}}|d^*\xi_{\epsilon}|^2=\int_{B_{R/2}}|d\varphi_{\epsilon}|^2\leq C_p\int_{B_R}|d\varphi_{\epsilon}|^p=C_p\int_{B_R}|d^*\xi_{\epsilon}|^p
\end{equation}
for any $p \in (1,\infty)$. We now recall one of the central observations of \cite{BBO}:
\begin{claim} For each $1\leq p<\frac{n}{n-1}$, there exists $C_p$ (independent of $\epsilon$) such that
\begin{equation}\label{xipest}
\int_M|d^*\xi_{\epsilon}|^p\leq C_p.
\end{equation}
\end{claim}
Once the claim is established, we'll obtain from (\ref{ellipregest}) a uniform bound
$$\int_{B_{R/2}(x)}|d^*\xi_{\epsilon}|^2\leq C$$
independent of $\epsilon$, and as a result,
\begin{equation}
\mu(B_{R/2}(x))\leq\lim_{j\to\infty}\frac{1}{|\log\epsilon_j|}\int_{B_{R/2}(x)}\frac{1}{2}|d^*\xi_{\epsilon_j}|^2=0,
\end{equation}
which is precisely what we needed to complete the proof of Lemma \ref{denslem}.

\hspace{6mm} To prove the claim, we follow closely the arguments in Section VI and the Appendix of \cite{BBO}. Fix $p \in [1,\frac{n}{n-1})$, and denote by $q=\frac{p}{p-1}>n$ its H\"{o}lder conjugate. The $L^p$ norm of $d^*\xi_{\epsilon}$ is then given by
\begin{equation}\label{lpnorm}
\|d^*\xi_{\epsilon}\|_{L^p(M)}=\sup\{\int_M\langle d^*\xi_{\epsilon},\beta\rangle \mid \beta \in \Omega^1(M),\text{ }\|\beta\|_{L^q}=1\}.
\end{equation}
Since $d^*\xi_{\epsilon}$ is co-exact, for any $\beta \in \Omega^1(M)$, we have that
\begin{equation}\label{pairing}
\int_M\langle d^*\xi_{\epsilon},\beta\rangle=\int_M\langle d^*\xi_{\epsilon},d^*d\alpha\rangle=\int_M\langle dd^*\xi_{\epsilon},d\alpha\rangle,
\end{equation}
where $\alpha:=\Delta_H^{-1}\beta$ is the unique solution of $\Delta_H\alpha=\beta$. (If we allowed $H^1_{dR}(M)\neq 0$, we would of course have to carry this out on the orthogonal complement of the harmonic one-forms.) Now, it follows from the $L^q$ regularity theory of the Hodge Laplacian (see \cite{Sc} for a careful treatment) that
$$\|d\alpha\|_{W^{1,q}}\leq C_q\|\beta\|_{L^q},$$
and since $q>n$, the Sobolev inequality yields
\begin{equation}\label{linftybound}
\|d\alpha\|_{L^{\infty}}\leq C_q\|\beta\|_{L^q}.
\end{equation}
Recalling that 
$$dd^*\xi_{\epsilon}=d\gamma_{\epsilon},$$
we can combine (\ref{lpnorm})-(\ref{linftybound}) to obtain the bound
\begin{equation}\label{l1control}
\|d^*\xi_{\epsilon}\|_{L^p(M)}\leq C_p\|d\gamma_{\epsilon}\|_{L^1}.
\end{equation}
All that remains is to bound the $L^1$ norm of $d\gamma_{\epsilon}$ as in \cite{BBO}, and this is straightforward. As we noted earlier, it follows from the definition of $f$ that 
$$\gamma_{\epsilon}=j(\frac{u_{\epsilon}}{|u_{\epsilon}|})\text{ is closed on }\{|u_{\epsilon}|^2\geq \frac{3}{4}\},$$
so that
$$spt(d\gamma_{\epsilon})\subset \{|u_{\epsilon}|^2\leq \frac{3}{4}\},$$
and by (\ref{volbounds}),
\begin{equation}\label{sptbounds}
|spt(d\gamma_{\epsilon})|\leq C\epsilon^2.
\end{equation}
Next, noting that
\begin{eqnarray*}
|d\gamma_{\epsilon}|&=&|f'(|u_{\epsilon}|^2)d|u_{\epsilon}|^2\wedge ju_{\epsilon}+f(|u_{\epsilon}|^2)dju_{\epsilon}|\\
&\leq &C|du_{\epsilon}^1\wedge du_{\epsilon}^2|,
\end{eqnarray*}
we conclude from (\ref{bochest}) that
\begin{equation}\label{dgsupbounds}
|d\gamma_{\epsilon}|\leq \frac{C}{\epsilon^2}
\end{equation}
pointwise. Combining (\ref{sptbounds}) and (\ref{dgsupbounds}), we obtain the desired $L^1$ bound
\begin{equation}\label{dgl1bound}
\|d\gamma_{\epsilon}\|_{L^1}\leq C,
\end{equation}
which together with (\ref{l1control}) completes the proof of the claim. The conclusion of Lemma \ref{denslem} and, consequently, Theorem \ref{thm2} then follow.

\end{document}